\documentclass[reqno,12pt]{amsart}
\usepackage{amssymb,amsmath,amsthm}
\usepackage{cite}
\usepackage{mathdots}
\usepackage{dsfont}
\usepackage{tikz}
\usepackage{float}
\usepackage{bm}

\def\C{\mathbb{C}}
\def\Q{\mathbb{Q}}
\def\R{\mathbb{R}}
\def\Z{\mathbb{Z}}

\def\Lt{\widetilde{L}}
\def\Pt{\widetilde{P}}
\def\Qt{\widetilde{Q}}

\def\Cb{\mathbb{C}}
\def\Nb{\mathbb{N}}
\def\Qb{\mathbb{Q}}
\def\Rb{\mathbb{R}}

\def\at{\widetilde{a}}
\def\ft{\widetilde{f}}

\def\Ft{\widetilde{F}}
\def\Lt{\widetilde{L}}

\def\deltat{\widetilde{\delta}}

\def\Pt{\widetilde{P}}
\def\Qt{\widetilde{Q}}

\def\Hc{\mathcal{H}}

\DeclareMathOperator\re{Re}

\newtheorem{theorem}{Theorem}
\newtheorem{lemma}[theorem]{Lemma}

\newtheorem{conjecture}[theorem]{Conjecture}

\theoremstyle{remark}

\newtheorem*{example}{Example}

\numberwithin{equation}{section}

\begin{document}


\title[Rational Approximations via Hankel Determinants]{Rational Approximations via Hankel Determinants}
\author{Timothy Ferguson}

\begin{abstract}
Define the monomials $e_n(x) := x^n$ and let $L$ be a linear functional. In this paper we describe a method which, under specified conditions, produces approximations for the value $L(e_0)$ in terms of Hankel determinants constructed from the values $L(e_1), L(e_2), \dots$. Many constants of mathematical interest can be expressed as the values of integrals. Examples include the Euler-Mascheroni constant $\gamma$, the Euler-Gompertz constant $\delta$, and the Riemann-zeta constants $\zeta(k)$ for $k \ge 2$. In many cases we can use the integral representation for the constant to construct a linear functional for which $L(e_0)$ equals the given constant and $L(e_1), L(e_2), \dots$ are rational numbers. In this case, under the specified conditions, we obtain rational approximations for our constant. In particular, we execute this procedure for the previously mentioned constants $\gamma$, $\delta$, and $\zeta(k)$. We note that our approximations are not strong enough to study the arithmetic properties of these constants.
\end{abstract}

\maketitle




\section{Introduction and Main Theorem} \label{sec:intro}

Approximation of a given constant by a sequence of rational numbers is an important problem in number theory. If these approximations are ``good enough'' then they can be used to study the arithmetic properties of the constant and obtain transcendence/irrationality results. In the case that the constant appears in a sequence of integrals, Zeilberger and Zudilin \cite{2019arXiv191210381Z} have obtained methods to automate irrationality proofs. They do not construct the approximations but rather provide a context in which the strength of the already specified approximations can be proved. In our case, we will use Hankel determinants to construct approximations. Hankel determinants have many applications to approximation. Bugeaud, Han, Wen, and Yao \cite{MR3509933} used Hankel determinants to prove that a large class of numbers have irrationality exponent two. In addition, Krattenthaler, Rochev, V\"{a}\"{a}n\"{a}nen, and Zudilin \cite{MR2475693} used Hankel determinants when showing that certain $q$-exponential functions have non-quadratic values. Furthermore, Zudilin \cite{MR3619445} used Hankel determinants to obtain a new irrationality criterion. Finally, we note that although evaluating Hankel determinants can be difficult, there are many methods for doing so (see \cite{MR3005311,MR2368911,MR1701596,MR2178686}).

Now we describe the setup for our method. Suppose that $L$ is a linear functional whose domain contains all polynomials. We show that if $L$ satisfies certain conditions, then we can construct approximations for $L(e_0)$ from $L(e_1)$, $L(e_2), \dots$ where $\{e_n\}_{n=0}^\infty$ are the monomials defined by $e_n(x) := x^n$. In particular, we will define two sequences of Hankel determinants $\{P_n\}_{n=0}^\infty$ and $\{Q_n\}_{n=0}^\infty$ such that
\begin{enumerate}
\item $\lim_{n \rightarrow \infty} P_n/Q_n = L(e_0)$,
\item $P_n, Q_n \in \Z[L(e_1), \dots, L(e_{2n+2})]$ for all $n \ge 0$.
\end{enumerate}
As a special case, we obtain rational approximations for $L(e_0)$ if $L(e_n) \in \Q$ for all $n \ge 1$. We will use this observation to construct rational approximations for the Euler-Mascheroni constant $\gamma$, the Euler-Gompertz constant $\delta$, and the Riemann-zeta constants $\zeta(k)$ where
\begin{align} \label{eq:gamma}
\gamma &:= \lim_{n \rightarrow \infty} \sum_{i=1}^n \frac{1}{i} - \log n = \int_0^1 \left( \frac{1}{1-x} + \frac{1}{\log x}\right)dx,
\\ \label{eq:delta}
\delta &:= \int_0^\infty \frac{e^{-x}}{1+x} dx,
\\ \label{eq:zeta}
\zeta(k) &:= \sum_{n=1}^\infty \frac{1}{n^k} = \frac{1}{(k-1)!} \int_0^\infty \frac{x^{k-1}}{e^x-1} dx, \quad k \ge 2.
\end{align}
We again note that our rational approximations are not strong enough to study the arithmetic nature of any these constants.

We now state our main theorem.
\begin{theorem} \label{thm:main}
Let $L : \Hc \rightarrow \R$ be a linear functional where $\Hc$ is a real vector space containing $\{e_n\}_{n=0}^\infty$. Define the sequences of Hankel determinants $\{P_n\}_{n=0}^\infty$ and $\{Q_n\}_{n=0}^\infty$ by
\begin{align*}
P_n := -\det(a_{i+j})_{i,j=0}^{n+1} \quad \text{and} \quad Q_n := \det(a_{i+j+2})_{i,j=0}^n \quad \text{where} \quad a_n :=
\begin{cases}
L(e_n) & \mbox{if $n \ge 1$,}
\\
0 & \mbox{if $n = 0$,}
\end{cases}
\end{align*}
Suppose that $L((e_1 f)^2) > 0$ for all non-zero real polynomials $f$. Then $Q_n > 0$ and
\begin{align}
L(e_1)^2/L(e_2) = P_0/Q_0 \le \dots \le P_n/Q_n \le P_{n+1}/Q_{n+1} \le \dots
\end{align}
In addition, suppose that $\Hc$ is a Hilbert space with inner product $\langle f, g \rangle := L(fg)$ and that $\{e_n\}_{n=1}^\infty$ is complete in $\Hc$ i.e. if $f \in \Hc$ satisfies $\langle f,e_n \rangle = 0$ for all $n \ge 1$, then $f = 0$. Then $P_n/Q_n < L(e_0)$ and
\begin{align} \label{eq:limit}
\lim_{n \rightarrow \infty} P_n/Q_n = L(e_0).
\end{align}
\end{theorem}

Our main theorem and its applications are exercises in Hilbert space theory (see \cite{MR949693} for relevant results). We defer the proof of Theorem \ref{thm:main} to the appendix. In Section \ref{sec:examples} we apply Theorem \ref{thm:main} to construct rational approximations for the constants $\gamma$, $\delta$, and $\zeta(k)$ in \eqref{eq:gamma}, \eqref{eq:delta}, and \eqref{eq:zeta}.



\section{Examples} \label{sec:examples}

In this section we apply Theorem \ref{thm:main} to the constants $\gamma$, $\delta$, and $\zeta(k)$ for $k \ge 2$. In each case we require Lemma \ref{lem:hilbert} to check the hypothesis of Theorem \ref{thm:main}. We defer the proof of Lemma \ref{lem:hilbert} to the appendix.

\begin{lemma} \label{lem:hilbert}
Let $\Omega$ be an interval of the form $[a,b]$ or $[a,\infty)$ and let $K : \Omega \rightarrow \R$ be a function satisfying the following properties:
\begin{enumerate}
\item $K$ is Lebesgue integrable and positive almost everywhere,
\item If $\Omega = [a,\infty)$, then there exists an $\epsilon > 0$ such that $K(y) = O(e^{-\epsilon y})$ as $y \rightarrow \infty$.
\end{enumerate}
Let $\Hc$ be the set of Lebesgue measurable functions $f : \Omega \rightarrow \R$ such that $\int_\Omega f(y)^2 K(y) dy < \infty$. Then $\Hc$ satisfies the following properties:
\begin{enumerate}
\item $\Hc$ is a Hilbert space with inner product $\langle f,g \rangle := \int_\Omega f(y) g(y) K(y) dy$,
\item $\Hc$ contains $\{e_0 \}_{n=0}^\infty$,
\item $\{e_n\}_{n=m}^\infty$ is complete in $\Hc$ for every $m \ge 0$.
\end{enumerate}
\end{lemma}

Now we begin our examples.

\begin{example}[Euler-Mascheroni constant $\gamma$]
Define the linear functional
\begin{align*}
L_\gamma(f) := & \int_0^1 \left( \frac{1}{x} + \frac{1}{\log (1-x)} \right) f(-x \log (1-x)) dx.
\end{align*}
Then $L_\gamma(e_0) = \gamma$ and
\begin{align*}
L_\gamma(e_n) = (-1)^n \int_0^1 x^{n-1} (\log (1-x))^n + x^n (\log (1-x))^{n-1} dx = (n-1)! \sum_{i=0}^n \binom{n}{i} (-1)^i \frac{n-2i-1}{(i+1)^{n+1}} \in \Q
\end{align*}
for all $n \ge 1$. Therefore $P_n, Q_n \in \Q$.

We note that the function $\sigma(x) := -x \log (1-x)$ is smooth and strictly increasing on the interval $[0,1)$ since $\sigma'(x) = -\log(1-x) + x/(1-x)$ is the sum of positive terms. Therefore we can make the change of variables $y = \sigma(x)$ and conclude that
\begin{align*}
L_\gamma(f) = \int_0^\infty f(y) K_\gamma(y) dy \quad \text{where} \quad K_\gamma(y) := \left( \frac{1}{\sigma^{-1}(y)}  + \frac{1}{\log(1-\sigma^{-1}(y))} \right) \frac{1-\sigma^{-1}(y)}{\sigma^{-1}(y) - (1-\sigma^{-1}(y))\log \sigma^{-1}(y)}.
\end{align*}
Since $1/x+1/\log(1-x)$ is integrable and positive on $[0,1]$ we see that $K_\gamma$ is integrable and positive on $\Omega_\gamma := [0,\infty)$.

Therefore we can apply Lemma \ref{lem:hilbert} as soon as we check the asymptotics of $K(y)$ as $y \rightarrow \infty$. Since $\lim_{y \rightarrow \infty} \sigma^{-1}(y) = 1$ we see that
\begin{align*}
\lim_{y \rightarrow \infty}  1/\sigma^{-1}(y)  + 1/\log(1- \sigma^{-1}(y)) = \lim_{y \rightarrow \infty} \sigma^{-1}(y) - (1-\sigma^{-1}(y))\log \sigma^{-1}(y) = 1.
\end{align*}
Therefore $K_\gamma(y)$ has the same asymptotics as $1-\sigma^{-1}(y)$. To determine the asymptotics of $\sigma^{-1}(y)$ we set $x = \sigma^{-1}(y) = 1 - e^{-y} R(y)$ and note that $0 \le e^{-y} R(y) \le 1$. Now plugging this ansatz into the equation $y = \sigma(x) = -x \log(1-x)$ we get the inequality
\begin{align*}
\log R(y) = -\frac{ye^{-y} R(y)}{1- e^{-y} R(y)} \le 0
\end{align*}
from which we conclude that $0 \le R(y) \le 1$. Therefore $1 - \sigma^{-1}(y) = O(e^{-y})$ hence $K_\gamma(y) = O(e^{-y})$ as $y \rightarrow \infty$. Therefore we can apply Lemma \ref{lem:hilbert} and make the corresponding definition of $\Hc_\gamma$.

Therefore by Theorem \ref{thm:main} we conclude that $0 < P_n/Q_n < \gamma$ and
\begin{align*}
0 < P_0/Q_0 \le \dots \le P_n/Q_n \uparrow \gamma \quad \text{as} \quad n \rightarrow \infty.
\end{align*}
We display data for our approximations $P_n/Q_n$ in Table \ref{tab:gammagompertz}.
\end{example}

\begin{example}[Euler-Gompertz constant $\delta$]
Define the linear functional
\begin{align*}
L_\delta(f) := \int_0^\infty f(x+1) \frac{e^{-x}}{x+1}dx.
\end{align*}
Then $L_\delta(e_0) = \delta$ and
\begin{align*}
L_\delta(e_n) = \int_0^\infty (x+1)^{n-1} e^{-x} dx = \sum_{i=0}^{n-1} \frac{(n-1)!}{i!} \in \Z
\end{align*}
for all $n \ge 1$. Therefore $P_n, Q_n \in \Z$.

With the change of variables $y = x+1$ we get that
\begin{align*}
L_\delta(f) = \int_1^\infty f(y) K_\delta(y) dy \quad \text{where} \quad K_\delta(y) := \frac{e^{-y+1}}{y}.
\end{align*}
Clearly $K_\delta$ satisfies the hypothesis of Lemma \ref{lem:hilbert} with $\Omega_\delta = [1,\infty)$, and we can define $\Hc_\delta$ as in Lemma \ref{lem:hilbert}. Therefore by Theorem \ref{thm:main} we conclude that $0 < P_n/Q_n < \delta$ and
\begin{align*}
0 < P_0/Q_0 \le \dots \le P_n/Q_n \uparrow \delta \quad \text{as} \quad n \rightarrow \infty.
\end{align*}
We display data for our approximations $P_n/Q_n$ in Table \ref{tab:gammagompertz}.
\end{example}

\begin{example}[Riemann-zeta constants $\zeta(k)$]
For $k \ge 2$ define the linear functional
\begin{align*}
L_{\zeta,k}(f) := \frac{1}{(k-1)!} \int_0^\infty \frac{x^{k-1}}{e^x-1} f(1-e^{-x})dx.
\end{align*}
Then $L_{\zeta,k}(e_0) = \zeta(k)$ and
\begin{align*}
L_{\zeta,k}(e_n) = \frac{1}{(k-1)!} \int_0^\infty x^{k-1} (1-e^{-x})^{n-1} e^{-x} dx = \sum_{i=0}^{n-1} \binom{n-1}{i} \frac{(-1)^i}{(i+1)^k} \in \Q
\end{align*}
for all $n \ge 1$. Therefore $P_n, Q_n \in \Q$. We also note that
\begin{align*}
L_{\zeta,k}(e_n) = \underbrace{\int_0^1 \dots \int_0^1}_{\text{$k$ times}} \frac{(1 - x_1 \dots x_k)^n}{1 - x_1 \dots x_k} dx_1 \dots dx_k
\end{align*}
for all $n \ge 0$.

With the change of variable $y = 1-e^{-x}$ we get that
\begin{align*}
L_{\zeta,k}(f) = \int_0^1 f(y) \frac{(-\log(1-y))^{k-1}}{(k-1)! y}dy = \int_0^1 f(y) K_{\zeta,k}(y)dy \quad \text{where} \quad K_{\zeta,k}(y) := \frac{(-\log(1-y))^{k-1}}{(k-1)! y}.
\end{align*}
Clearly $K_{\zeta,k}$ satisfies the hypothesis of Lemma \ref{lem:hilbert} with $\Omega_{\zeta,k} := [0,1]$, and we can define $\Hc_{\zeta,k}$ as in Lemma \ref{lem:hilbert}. Therefore by Theorem \ref{thm:main} we conclude that $0 < P_n/Q_n < \zeta(k)$ and
\begin{align*}
0 < P_0/Q_0 \le \dots \le P_n/Q_n \uparrow \zeta(k) \quad \text{as} \quad n \rightarrow \infty.
\end{align*}
We display data for our approximations $P_n/Q_n$ in Table \ref{tab:zeta}.
\end{example}

\begin{table}[H]
\renewcommand{\arraystretch}{1.5}
\begin{tabular}{|c|c|c|c|c|}
\hline
$n$ & \multicolumn{2}{|c|}{$P_n/Q_n$ for $\gamma$} & \multicolumn{2}{|c|}{$P_n/Q_n$ for $\delta$}
\\
\hline
0 & $\frac{9}{41}$ & 0.2195121951 & $\frac{1}{2}$ & 0.5000000000
\\
\hline
1 & $\frac{627726506}{2084484569}$ & 0.3011423137 & $\frac{4}{7}$ & 0.5714285714
\\
\hline
2 & - & 0.3457225856 & $\frac{10}{17}$ & 0.5882352941
\\
\hline
3 & - & 0.3745360864 & $\frac{124}{209}$ & 0.5933014354
\\
\hline
4 & - & 0.3950172588 & $\frac{460}{773}$ & 0.5950840880
\\
\hline
5 & - & 0.4104941483 & $\frac{7940}{13327}$ & 0.5957829969
\\
\hline
6 & - & 0.4226993663 & $\frac{39020}{65461}$ & 0.5960801088
\\
\hline
7 & - & 0.4326321010 & $\frac{859580}{1441729}$ & 0.5962146839
\\
\hline
8 & - & 0.4409129928 & $\frac{748420}{1255151}$ & 0.5962788541
\\
\hline
9 & - & 0.4479499436 & $\frac{139931620}{234662231}$ & 0.5963107885
\\
\hline
10 & - & 0.4540232182 & $\frac{1015353820}{1702678841}$ & 0.5963272671
\\
\hline
11 & - & 0.4593324215 & $\frac{31805257340}{53334454417}$ & 0.5963360400
\\
\hline
12 & - & 0.4640239850 & $\frac{267257395340}{448162154317}$ & 0.5963408395
\\
\hline
13 & - & 0.4682080352 & $\frac{9591325648580}{16083557845279}$ & 0.5963435293
\\
\hline
14 & - & 0.4719691667 & $\frac{8317039567460}{13946689584823}$ & 0.5963450693
\\
\hline
15 & - & 0.4753735569 & $\frac{75451991521660}{126523856174033}$ & 0.5963459683
\\
\hline
17 & - & 0.4813123036 & $\frac{160957871380291180}{269906478537389909}$ & 0.5963468245
\\
\hline
19 & - & 0.4863363761 & $\frac{60588676286095139260}{101599675414361566913}$ & 0.5963471442
\\
\hline
21 & - & 0.4906573284 & $\frac{714785218276618032951940}{1198605668577020653881647}$ & 0.5963472700
\\
\hline
23 & - & 0.4944242261 & - & 0.5963473218
\\
\hline
25 & - & 0.4977454856 & - & 0.5963473439
\\
\hline
\end{tabular}
\renewcommand{\arraystretch}{1}
\caption{\label{tab:gammagompertz} Approximants $P_n/Q_n$ for the Euler-Mascheroni constant $\gamma = 0.5772156649...$ and the Euler-Gompertz constant $\delta = 0.5963473623...$.}
\end{table}

\begin{table}[H]
\renewcommand{\arraystretch}{1.5}
\begin{tabular}{|c|c|c|c|c|}
\hline
$n$ & \multicolumn{2}{|c|}{$P_n/Q_n$ for $\zeta(2)$} & \multicolumn{2}{|c|}{$P_n/Q_n$ for $\zeta(3)$}
\\
\hline
0 & $\frac{4}{3}$ & 1.333333333 & $\frac{8}{7}$ & 1.142857143
\\
\hline
1 & $\frac{135}{89}$ & 1.516853933 & $\frac{4887}{4105}$ & 1.190499391
\\
\hline
2 & $\frac{505319}{320733}$ & 1.575512966 & $\frac{13305034871}{11102509809}$ & 1.198380826
\\
\hline
3 & $\frac{1337517425}{835187004}$ & 1.601458618 & $\frac{2196507603137550625}{1829598054203124216}$ & 1.200541069
\\
\hline
4 & $\frac{26920197674520019}{16667096529827700}$ & 1.615170202 & - & 1.201321520
\\
\hline
5 & $\frac{4108034695656989506227}{2530690380879633004100}$ & 1.623286170 & - & 1.201657975
\\
\hline
6 & - & 1.628483935 & - & 1.201822087
\\
\hline
7 & - & 1.632011765 & - & 1.201909799
\\
\hline
8 & - & 1.634515372 & - & 1.201960105
\\
\hline
9 & - & 1.636356043 & - & 1.201990623
\\
\hline
10 & - & 1.637748743 & - & 1.202010004
\\
\hline
11 & - & 1.638827873 & - & 1.202022790
\\
\hline
12 & - & 1.639680964 & - & 1.202031499
\\
\hline
13 & - & 1.640367005 & - & 1.202037598
\\
\hline
14 & - & 1.640926928 & - & 1.202041971
\\
\hline
15 & - & 1.641389854 & - & 1.202045173
\\
\hline
17 & - & 1.642103939 & - & 1.202049371
\\
\hline
19 & - & 1.642622098 & - & 1.202051847
\\
\hline
21 & - & 1.643009963 & - & 1.202053385
\\
\hline
23 & - & 1.643307821 & - & 1.202054380
\\
\hline
25 & - & 1.643541511 & - & 1.202055046
\\
\hline
\end{tabular}
\renewcommand{\arraystretch}{1}
\caption{\label{tab:zeta} Approximants for the Riemann-zeta constants $\zeta(2) = 1.644934067...$ and $\zeta(3) = 1.202056903...$.}
\end{table}

We finish this section with an example which only satisfies the first assumption of Theorem \ref{thm:main} but not the second set of assumptions. We then verify that the conclusions after the first assumption hold but not the conclusions after the  second set of assumptions. Define the linear functional
\begin{align*}
\Lt_\gamma(f) := -\int_0^\infty (\log x) \frac{d}{dx} (f(x) e^{-x}) dx.
\end{align*}
Then $\Lt_\gamma(e_0) = \int_0^\infty (\log x) e^{-x} dx = -\gamma$ and
\begin{align*}
\Lt_\gamma(e_n) = -\int_0^\infty (\log x) \frac{d}{dx} (x^n e^{-x}) dx = -(\log x) x^n e^{-x} \biggr\rvert_0^\infty + \int_0^\infty x^{n-1} e^{-x} dx = (n-1)!
\end{align*}
for all $n \ge 1$. Similarly,
\begin{align*}
\Lt_\gamma((e_1f)^2) &= -\int_0^\infty (\log x) \frac{d}{dx} (x^2 f(x)^2 e^{-x}) dx
\\
&= -(\log x)x^2f(x)^2 e^{-x} \biggr\rvert_0^\infty + \int_0^\infty xf(x)^2 e^{-x} dx = \int_0^\infty x f(x)^2 e^{-x} dx > 0
\end{align*}
for all non-zero real polynomials $f$.

Therefore the first assumption of Theorem \ref{thm:main} holds. Now using the formula for $\Lt_\gamma(e_n)$ for $n \ge 1$ we check via computer that $P_n, Q_n > 0$ and $P_n/Q_n = \sum_{i=1}^{n+1} 1/i$ for $0 \le n \le 50$. (We conjecture that this identity and inequalities hold for all $n \ge 0$ but do not prove it.) This numerically verifies that the conclusions following the first assumption in Theorem \ref{thm:main} hold.

But the second set of assumptions of Theorem \ref{thm:main} fail. This is because $\langle e_0, e_0 \rangle = \Lt_\gamma(e_0^2) = -\gamma < 0$ so that $\langle f,g \rangle = \Lt_\gamma(fg)$ cannot define an inner product on any vector space $\Hc$ containing $\{e_n\}_{n=0}^\infty$. Based off of our conjecture we clearly deduce that $P_n/Q_n \ge \Lt_\gamma(e_0)$ for all $n \ge 0$ and that $\lim_{n \rightarrow \infty} P_n/Q_n$ doesn't even exist since the harmonic series diverges. This numerically verifies that the conclusions following the second set of assumptions in Theorem \ref{thm:main} fail.



\section{Conclusion} \label{sec:conclusion}

We described a method by which the values $L(e_1), L(e_2), \dots$ of a linear functional can be used to approximate $L(e_0)$ under specified conditions. In particular, we use the values $L(e_1), L(e_2), \dots$ to construct two sequences of Hankel determinants $P_n$ and $Q_n$ for which $\lim_{n \rightarrow \infty} P_n/Q_n = L(e_0)$. We then applied our method to construct rational approximations for the Euler-Mascheroni constant $\gamma$, the Euler-Gompertz constant $\delta$, and the Riemann-zeta constants $\zeta(k)$ for $k \ge 2$.



\section{Acknowledgements}

The author gratefully acknowledges support as a postdoctoral associate at Arizona State University.



\section{Appendix}

In this section we prove Theorem \ref{thm:main} and Lemma \ref{lem:hilbert}. But before we prove Theorem \ref{thm:main} we first need Lemma \ref{lem:det} which we now state and prove.

\begin{lemma} \label{lem:det}
Let $A_n = (a_{i,j})_{i,j=0}^n$ be a matrix such that $a_{i,j} = 0$ if $i \ne j$ and $i,j > 0$. Then
\begin{align*}
\det(A_n) = \left( \prod_{i=1}^n a_{i,i} \right)\left(a_{0,0} - \sum_{i=1}^n \frac{a_{i,0} a_{0,i}}{a_{i,i}} \right).
\end{align*}
\end{lemma}

\begin{proof}[Proof of Lemma \ref{lem:det}]
The result trivially holds for $n = 0$ so suppose that it holds for some $n \ge 0$. Then expanding by cofactors along the last row and then the last column of $A_{n+1}$ we get that
\begin{align*}
\det(A_{n+1}) = a_{n+1,n+1} \det(A_n) - a_{n+1,0} a_{0,n+1} \prod_{i=0}^n a_{i,i} = \left( \prod_{i=1}^{n+1} a_{i,i} \right)\left(a_{0,0} - \sum_{i=1}^{n+1} \frac{a_{i,0} a_{0,i}}{a_{i,i}} \right)
\end{align*}
by the inductive hypothesis.
\end{proof}

\begin{proof}[Proof of Theorem \ref{thm:main}]
Define the modified Hankel determinants
\begin{align*}
\Pt_n := -\det(L(e_{i+j}))_{i,j=0}^{n+1} = P_n - L(e_0) Q_n \quad \text{and} \quad \Qt_n := \det(L(e_{i+j+2}))_{i,j=0}^n = Q_n.
\end{align*}
Now since $L(e_2 f^2) > 0$ for all non-zero real polynomial $f$ we know that the matrix $(L( e_{i+j+2}))_{i,j=0}^n$ is positive definite. Therefore $\Qt_n = Q_n > 0$ for all $n \ge 0$. It is well-known that this guarantees the existence of a sequence of polynomials $\{ q_n \}_{n=0}^\infty$ where $q_n$ is a monic polynomial of degree $n$ and
\begin{align*}
L(e_2 q_i q_j)
\begin{cases}
> 0 & \mbox{if $i = j$,}
\\
= 0 & \mbox{if $i \ne j$.}
\end{cases}
\end{align*}
Now define the new sequence of polynomials $\{ p_n \}_{n=0}^\infty$ by
\begin{align*}
p_n :=
\begin{cases}
e_1 q_{n-1} & \mbox{if $n \ge 1$,}
\\
e_0 & \mbox{if $n = 0$.}
\end{cases}
\end{align*}
Note again that $p_n$ is a monic polynomial of degree $n$ and that $L(p_ip_j) = 0$ if $i \ne j$ and $i,j > 0$. Therefore
\begin{align*}
\Qt_n = \det(L(e_{i+j+2}))_{i,j=0}^n = \det(L(e_2 q_i q_j))_{i,j=0}^n = \prod_{i=0}^n L(e_2 q_i^2) = \prod_{i=1}^{n+1} L(p_i^2)
\end{align*}
and by Lemma \ref{lem:det}
\begin{align*}
\Pt_n = -\det(L(e_{i+j}))_{i,j=0}^{n+1} = -\det(L(p_i p_j))_{i,j=0}^{n+1} = -\Qt_n \left( L(e_0) - \sum_{i=1}^{n+1} \frac{L(p_i)^2}{L(p_i^2)} \right)
\end{align*}
from which we conclude that
\begin{align*}
\frac{P_n}{Q_n} = \frac{\Pt_n}{\Qt_n} + L(e_0) = \sum_{i=1}^{n+1} \frac{L(p_i)^2}{L(p_i^2)} \ge 0
\end{align*}
is monotonically increasing. This completes the first part of the proof.

For the second part, we note that $\{p_n/\sqrt{L(p_n^2)} \}_{n=1}^\infty$ is an orthonormal basis for $\Hc$ since $\{ e_n\}_{n=1}^\infty$ is complete. Therefore
\begin{align*}
L(e_0) = \|e_0\|^2 = \sum_{i=1}^\infty \biggr\rvert \biggr\langle e_0, \frac{p_i}{\sqrt{L(p_i^2)}} \biggr\rangle \biggr\rvert^2 = \sum_{i=1}^\infty \frac{L(p_i)^2}{L(p_i^2)}
\end{align*}
by Parseval's identity. Now suppose that $P_n/Q_n = L(e_0)$ for some $n \ge 0$. Then $L(p_i) = 0$ for $i > n+1$ and
\begin{align*}
e_0 = \sum_{i=1}^\infty \biggr\langle e_0, \frac{p_i}{\sqrt{L(p_i^2)}} \biggr\rangle \frac{p_i}{\sqrt{L(p_i^2)}} = \sum_{i=1}^\infty \frac{L(p_i)}{L(p_i^2)} p_i = \sum_{i=1}^{n+1} \frac{L(p_i)}{L(p_i^2)} p_i
\end{align*}
which is a contradiction since each $p_i$ is divisible by $e_1$ for $i \ge 1$.
\end{proof}

\begin{proof}[Proof of Lemma \ref{lem:hilbert}]
We only prove (3) since (1) and (2) follow by standard arguments. To demonstrate that $\{e_n\}_{n=m}^\infty$ is dense in $\Hc$ we show that $f = 0$ is the only $f \in \Hc$ such that $\langle f, e_n \rangle = 0$ for all $n \ge m$. Now for any $z$ (we require $|z| < \epsilon/2$ if $\Omega$ is unbounded) and $k \ge 0$ we have that
\begin{align*}
\left( \int_\Omega |f(y)| y^k e^{|z|y} K(y) dy \right)^2 \le \left( \int_\Omega |f(y)|^2 K(y) dy \right) \left( \int_\Omega y^{2k} e^{2|z|y} K(y) dy \right) < \infty
\end{align*}
by the Cauchy-Schwarz inequality. Therefore by the dominated convergence theorem we have that
\begin{align*}
F(z) = \int_\Omega f(y) y^m e^{-zy} K(y) dy = \sum_{n=0}^\infty \frac{(-z)^n}{n!} \int_\Omega f(y) y^{n+m} K(y) dy = 0.
\end{align*}
If $\Omega$ is bounded, then $F(z) = 0$ for all $z$. If $\Omega$ is unbounded, then another application of the dominated convergence theorem shows that $F(z)$ is analytic for $\re z > 0$ so that $F(z) = 0$ for $\re z > 0$ by analytic continuation. Define $\ft : [0,\infty) \rightarrow \R$ by
\begin{align*}
\ft(y) :=
\begin{cases}
f(y+a) (y+a)^m e^{-ay} K(y+a) & \mbox{if $y + a \in \Omega$,}
\\
0 & \mbox{if $y + a \notin \Omega$.}
\end{cases}
\end{align*}
Then $\ft$ is an integrable function whose Laplace transform $\Ft(z) := \int_0^\infty \ft(y) e^{-zy} dy = 0$ for $\re z > 0$. Therefore $\ft$ and hence $f$ is zero almost everywhere.
\end{proof}



\bibliography{RationalApproximationsViaHankelDeterminants}
\bibliographystyle{plain}

\end{document}